\documentclass[11pt]{amsart}

\usepackage{amsmath}
\usepackage{amsfonts}
\usepackage{amssymb}
\usepackage{amsthm}
\usepackage{indentfirst}

\usepackage{tikz}
\usepackage{verbatim}
\usetikzlibrary{matrix}

\usepackage{enumerate}

\newtheorem{theorem}{Theorem}[section]

\newtheorem{proposition}[theorem]{Proposition}

\newtheorem{lem}[theorem]{Lemma}

\newtheorem{definition}[theorem]{Definition}

\newtheorem{assumption}[theorem]{Assumption}

\newtheorem{remark}[theorem]{Remark}
\newtheorem{notation}[theorem]{Notation}
\newtheorem{example}[theorem]{Example}
\newcommand{\bfF}{{\bf F}}
\newcommand{\bfV}{{\bf V}}
\newcommand{\Gal}{\operatorname{Gal}}
\newcommand{\ds}{\displaystyle}
\newcommand{\Tr}{\operatorname{Tr}}

\newcommand{\Sel}{\operatorname{Sel}}
\newcommand\Selr{\Sel_{\text{rel}}}
\newcommand{\Hom}{\operatorname{Hom}}
\newcommand{\End}{\operatorname{End}}

\newcommand\rat{{\mathbb Q}}

\newcommand\rank{\operatorname{rank}}

\newcommand\corank{\operatorname{corank}}

\newcommand\Proj{\operatorname{Proj}}
\newcommand\Qp{{\mathbb Q_p}}
\newcommand\Zp{{\mathbb Z_p}}

\newcommand\Z{\mathbb Z}

\newcommand{\bfA}{\mathbf A}

\newcommand{\Dieudonne}{\mathbf D}

\newcommand\mfp{\mathfrak p}
\newcommand\mfq{\mathfrak q}

\newcommand\mfA{\mathbf A}

\newcommand\mm{\mathfrak m}

\newcommand\bfx{\mathbf x}
\newcommand\bfy{\mathbf y}
\newcommand\bfm{\mathbf m}
\newcommand\bfl{\mathbf l}

\newcommand\Q{\mathbb Q}

\newcommand\C{\mathbb C}

\newcommand\bfL{\mathbf L}

\newcommand\OO{\mathcal O}

\newenvironment{nouppercase}{%
  \renewcommand{\uppercasenonmath}[1]{}}{}


\newcommand\Label{\label}

\title[Ranks of Jacobians]{Ranks of rational points of the Jacobian varieties of hyperelliptic curves}
\thanks{Bo-Hae Im is supported by Basic Science Research Program through the National Research Foundation of Korea (NRF) funded by the Ministry of Education (NRF-2014R1A1A2053748). }
\author{Bo-Hae Im}
\address{Department of Mathematical Sciences, KAIST, 291, Daehak-Ro, Yuseong-Gu, Daejeon, 34141, Republic of Korea}
\email{bhim@kaist.ac.kr}
\author{Byoung Du Kim}
\address{School of Mathematics and Statistics, Victoria University of Wellington, Wellington 6140, New Zealand}
\email{byoungdu.kim@vuw.ac.nz}

\date{\today}

\begin{document}

\begin{abstract}
In this paper, we obtain bounds for the Mordell-Weil ranks over cyclotomic extensions of a wide range of abelian varieties defined over a number field $F$ whose primes above $p$ are totally ramified over $F/\Q$. We assume that the abelian varieties may have good non-ordinary reduction at those primes. Our work is a generalization of \cite{Kim}, in which the second author generalized Perrin-Riou's Iwasawa theory for elliptic curves over $\Q$ with supersingular reduction (\cite{Perrin-Riou}) to elliptic curves defined over the above-mentioned number field $F$. On top of non-ordinary reduction and the ramification of the field $F$, we deal with the additional difficulty that the dimensions of the abelian varieties can be any number bigger than 1 which causes a variety of issues. As a result, we obtain bounds for the ranks over cyclotomic extensions $\Q(\mu_{p^{\max(M,N)+n}})$ of the Jacobian varieties of {\it ramified} hyperelliptic curves $y^{2p^M}=x^{3p^N}+ax^{p^N}+b$ among others.

\end{abstract}
\begin{nouppercase}
\maketitle
\end{nouppercase}

\begin{nouppercase}
\begin{section}{Introduction}
In this paper, we construct an (Iwasawa) theory in the spirit of \cite{Mazur} and \cite{Perrin-Riou}, and obtain bounds for the Mordell-Weil ranks of abelian varieties. Our first model is Barry Mazur's work, \cite{Mazur}, in which he studied abelian varieties $A/F$ with good ordinary reduction at every prime of $F$ above a prime $p$. Another model we follow closely is Perrin-Riou's work, \cite{Perrin-Riou}, in which she studied elliptic curves defined over $\Q$ with good supersingular reduction, and succeeded in overcoming many difficulties due to the supersingular reduction type.

To explain the relevance of their work to our work, suppose that $A$ is an abelian variety defined over a number field $F$, and $F_{\infty}$ is a $\Zp$-extension of $F$ (i.e., $\Gal(F_{\infty}/F)\cong \Zp$) for a prime $p$. When $A$ has good ordinary reduction at every prime of $F$ lying above $p$ (i.e., $A$ has good reduction at such a prime, and its associated formal group is of multiplicative type), the celebrated work (\cite{Mazur}) provides a criterion for whether $A(F_{\infty})$ has a finite rank or not. It is a great example of the potential strength of Iwasawa Theory.

Although a similar result is expected for abelian varieties generally regardless of the reduction type, how to prove it is not known in many cases if not in most cases. The main obstacle is that a formal group of non-multiplicative type does not admit a non-trivial universal norm.

One of the notable attempts in the case of the non-ordinary reduction was the aforementioned work of Perrin-Riou (\cite{Perrin-Riou}). 
Her insight was that we might be able to use Fontaine's theory of group schemes in a clever way to construct a series of local points which  satisfy a certain norm relation (but do not constitute a universal norm).

In particular, she obtained that for an elliptic curve $E/\Q$ with good supersingular reduction at $p$, if $\bfL_{\alpha}\not=0$ (which is the ``algebraic'' $p$-adic $L$-function she constructed), then

$$\corank_{\Zp} \Sel_p(E/\Q_n) \leq (p-1)(p^{n-1}+p^{n-2}+\cdots+ p^m)+C,$$ 
where $C$ is some fixed constant, and $n-m=\frac n2 +O(1)$. Most surprisingly, when $a_p(E)=1+p-E(\Z/p\Z)=0$, she applied a more refined idea, and obtained that $\rank (E/\Q_{\infty})$ is finite just as Mazur showed for $E/\Q$ with good ordinary reduction at $p$. Now, we have a more sophisticated Iwasawa theory of S. Kobayashi (\cite{Kobayashi}) for $E/\Q$ in the case of the cyclotomic $\Zp$-extension of $\Q$ (and also note F. Sprung's work, \cite{Sprung}, which generalizes Kobayashi's work, which applies to most elliptic curves defined over $\Q$, to every elliptic curve defined over $\Q$). The work close to our work in its subject matter is \cite{Kazim-Lei}  (which crucially relied on \cite{Lei-Loeffler-Zerbes}) which generalized the aforementioned results to abelian varieties in a context different from that of this paper.

Our goal is to generalize her work to abelian varieties over a number field $F$ whose primes above $p$ are totally ramified over $F/\Q$. We assume that the abelian varieties may have non-ordinary reduction at primes above $p$ (but otherwise, have good reduction at every prime above $p$). We use the ideas of the second author's earlier paper (\cite{Kim}), in which he studied abelian varieties of dimension 1 (hence elliptic curves). Three main obstacles are the reduction type, the ramification of the field, and the dimensions of the abelian varieties. The last one is a new issue not covered by \cite{Kim}, and as we will argue, it is not a trivial one.

In \cite{Kim}, the second author found that by applying Fontaine's theory more rigorously and judiciously, Perrin-Rou's idea can be extended to elliptic curves $A$ defined over totally ramified fields $F$ (with emphasis on ``\textit{ramified}''). In the manner of Perrin-Riou, he constructed a $p$-adic power series $\bfL_{\alpha}$ for a root $\alpha$ of the characteristic of the Dieudonne module of $A$ ($\bfL_{\alpha}$ becomes an integral power series if $\alpha$ is a unit), and showed that if it is not $0$, then 
$$\corank \Sel_p(A/F_n) \leq (p-1)(p^{n-1}+p^{n-2}+\cdots+p^m)+C,$$ 
where $\lambda=v(\alpha)$, and $n-m=\lambda n+O(1)$. Also, under some conditions, he generalized Kobayashi's theory to  the above elliptic curves $A$. It should be seen as an effort to establish an Iwasawa theory that works well in a more general case.

In this paper, we generalize the result in \cite{Kim} to abelian varieties of any dimension. (We keep the condition that every prime of $F$ above $p$ is totally ramified over $F/\Q$.) 
See Proposition~\ref{ZeroGo} for our main statement. Since we study abelian varieties of any dimension, we have the following issues: Our group schemes can have ``mixed reduction'' (i.e., a mix of ordinary and non-ordinary reduction), and different generators of the Dieudonne module have different ``minimal polynomials'', thus different series of local points we construct have different norm relations. This is all very different from abelian varieties of dimension 1 (i.e., elliptic curves) which has only two types of good reduction at a given prime: good ordinary, and good supersingular. And, their Dieudonne module is generated by one element over $\Dieudonne$. Thus, we resort to a rather complicated construction in Section~\ref{China}.

Our work involves constructing explicit logarithms, and as part of the work, we give an explicit and general definition of the constant term of the logarithm, without which the resulting local points do not satisfy the norm relations. One may define it as a number which happens to force the resulting local points to satisfy the norm relations, but our construction is more natural in the sense that it explains the rather mysterious existence of such constant terms.

As a result of our work, we obtain bounds for the Mordell-Weil ranks of abelian varieties in a wide range of cases. The bounds are given in terms of the dimension of the abelian varieties, and some other terms associated to their Dieudonne modules. In particular, for non-negative integers $N$ and $M$, we consider the Jacobian variety of the hyperelliptic curve,
$$C_N:y^2=x^{3p^N}+ax^{p^N}+b,$$ 
and the Jacobian variety of the curve, 
$$C_{M,N} : y^{2p^M}=x^{3p^N}+ax^{p^N}+b.$$

For the lack of better words, we call $C_{M,N}$ a {\it ramified hyperelliptic curve}. In the former, we suppose $C_N$ is defined over $F=\Q(\zeta_{p^N})$, and in the latter, over $F=\Q(\zeta_{p^M}, \zeta_{p^N})$. In both, we let $F_{\infty}=\Q(\zeta_{p^{\infty}})$, and $F_n$ be the field $F \subset F_n \subset F_{\infty}$ so that $\Gal(F_n/F) \cong \Z/p^n\Z$. We let $H$ be a number field so that every prime of $H$ above $p$ is unramified over $H/\Q$. (For example, we may assume $H$ is the field of complex multiplication if the curve has complex multiplication.)

Suppose $A$ is either the Jacobian variety of the curve above, or its dual abelian variety. We choose certain irreducible polynomials $q_1(x),\ldots, q_s(x)$ ($s=\dim A \cdot [HF:\Q]$) which are associated to the Dieudonne modules of $A$, and choose a zero $\alpha_i$ of each $q_i(x)$. (See the discussion after Assumption~\ref{Tyrol}.) And, we construct the algebraic $p$-adic $L$-function $\bfL_{ \{ \alpha_i \}_i}$ (see Definition~\ref{Madrid}). Its construction is necessarily more sophisticated than that of $\bfL_{\alpha}$ in \cite{Kim} because the Dieudonne module in this case has a higher dimension. In many ways, constructing $\bfL_{ \{ \alpha_i \}_i}$ is pivotal to our study. Then we obtain:

\begin{theorem}		\Label{Vienna}
Let $\bfA'=\Hom(\cup_n A[p^n], \Zp(1))$ where $\Zp(1)$ is the Tate twist of the trivial representation $\Zp$. (Equivalently, $\bfA'=\cup_n A^{\vee}[p^n]$ where $A^{\vee}$ is the dual abelian variety of $A$.) If $\bfL_{\{ \alpha_i \}_i} \not=0$, then for some fixed $C$,
$$ \corank_{\Zp} \Sel_p(\bfA'/H\cdot F_n) \leq \sum_{i=1}^s (p-1) \times \left( p^{n-1}+p^{n-2}+ \cdots+ p^{m_i} \right)+C,$$
as $n$ varies, where $n-m_i = \lambda_i n +O(1)$ and $\lambda_i=v_p(\alpha_i)$.

In particular,  let $\lambda$  be the maximum of $\lambda_1, \cdots, \lambda_s$. Then we have the following.

\begin{enumerate}[(a)]
\item If $A$ is the Jacobian variety $J_N/F$ of $C_N$, or its dual abelian variety $J_N^{\vee}$, then for some fixed $C$,
\begin{eqnarray*} &&\corank_{\Zp} \Sel_p(\bfA'/H\cdot F_n) \\&\leq& \frac{3p^N-1}{2} \cdot [H\cdot F_n:\Q] \cdot (p-1) \times \left( p^{n-1}+p^{n-2}+ \cdots+ p^m \right)+C,
\end{eqnarray*}
as $n$ varies, where $n-m = \lambda n +O(1)$.

\item If $A$ is the Jacobian variety $J_{M,N}$ of $C_{M,N}$, or its dual abelian variety $J_{M,N}^{\vee}$, then
		\begin{enumerate}
		\item[(i)] when $M>N$, for some fixed $C$,
\begin{eqnarray*} &&\corank_{\Zp} \Sel_p(\bfA'/H\cdot F_n) \\&\leq& \left( 1-2p^M+\frac{3p^N(2p^M-1)+p^M}{2} \right) \cdot [H\cdot F_n:\Q] \times	
 (p-1) \\
 & &\times \left( p^{n-1}+p^{n-2}+ \cdots+ p^m \right)+C,
\end{eqnarray*}
as $n$ varies, where $n-m = \lambda n +O(1)$, and

		\item[(ii)] when $M\leq N$,  for some fixed $C$,
\begin{eqnarray*} &&\corank_{\Zp} \Sel_p(\bfA'/H\cdot F_n)\\ &\leq& \left( 1-2p^M+\frac{3p^N(2p^M-1)+2p^M-p^N}{2} \right) \cdot [H\cdot F_n:\Q] \times (p-1)	\\
&&  \times \left\{ p^{n-1}+p^{n-2}+ \cdots+ p^m \right\}+C,
\end{eqnarray*}
as $n$ varies where $n-m = \lambda n +O(1)$.
		\end{enumerate}

\end{enumerate}
\end{theorem}

\vspace{1.5mm}

Theorem~\ref{Vienna} follows from Proposition~\ref{ZeroGo} and Lemma~\ref{Tripoli}. As noted in \cite{Kim}, $\bfL_{\{ \alpha_i \}} \not=0$ if $\Sel_p(\bfA'/H\cdot F_n)^{\chi}$ is finite for any $n$ and any character $\chi$ of $\Gal(H\cdot F_n/H\cdot F)$.

Note that $\End_F (J_N)$ contains $\Z[\zeta_{p^N}]$, and $\End_F (J_{M,N})$ contains $\Z[\zeta_{p^M}] \times \Z[\zeta_{p^N}]$. It is illustrative that the bounds in Theorem~\ref{Vienna}~(a) and (b) are proportional to the ranks of $\Z[\zeta_{p^N}]$ or $\Z[\zeta_{p^M}] \times \Z[\zeta_{p^N}]$. We note that the bounds in Theorem~\ref{Vienna} are probably better than the bounds obtainable by previously known Iwasawa theory methods in our case.

\end{section}


\begin{section}{Constructing local points of group schemes of higher dimension}

\Label{China}

For Fontaine's theory of group schemes over local fields, refer to the original reference \cite{Fontaine}, or \cite{Kim}. We use the notations and definitions the second author used in \cite{Kim}~Section~3 (and we do not repeat them). As in \cite{Kim}, we may drop $k$ from $\Dieudonne_k=W[\bfF, \bfV]$ if $k=\Z/p\Z$.

We recall the Dieudonne module $M=M(G)$ of a group scheme $G$ over a finite field $k$. Also recall that where $K'$ is a (possibly ramified) extension of $\Qp$, $G$ is a smooth group scheme over $\OO_{K'}$, and $G_{/k}$ is the reduced group scheme over the residue field $k$, $L \subset M(G_{/k})_{\OO_{K'}}$ is the set of logarithms of $G$.

\begin{subsection}{}			\Label{Ordinary part}
Let $K'$ be a totally ramified extension of $\Qp$, and $G$ be a smooth formal group scheme over $\OO_{K'}$ whose reduction is also smooth. Let $L$ and $M$ be respectively the set of logarithms of $G$, and the Dieudonne module of $G$.

Assume that $M$ is torsion-free. In this section, we study a method to generate local points of $G$.

Naturally, we will build on the ideas of \cite{Kim}, but working with formal groups of higher dimensions require new ideas. There are several issues. A lesser one is that $M$ and $L$ are not generated by one element. A bigger issue is that ordinary reduction types and non-ordinary reduction types are mixed in general, and cannot be easily separated to the best of our knowledge. The method for group schemes of dimension 1 would naturally extend to group schemes of any dimension if they are (over the given field) isogenous to products of group schemes of dimension 1. But every group scheme may not be of this kind. So, we need to develop a more sophisticated method.

Finally, unlike the dimension one case, each generator of $M$ may have a different ``minimal polynomial'' (which we define later), and the minimal polynomial may not be irreducible.

We will present our solutions to these issues. First we note the following:

\begin{proposition}		\Label{Optical Fiber}
$M\otimes \Qp \cong \Big(M^{ord} \otimes \Qp\Big) \times \Big( M^{non-ord} \otimes \Qp\Big)$ for some finitely generated $\Dieudonne$-modules $M^{ord}$ and $M^{non-ord}$ so that all the eigenvalues of $\bfF$ acting on $M^{ord}$ have valuation 1, and all the eigenvalues of $\bfF$ acting on $M^{non-ord}$ have valuation less than 1.
\end{proposition}

\begin{proof}
Since $\bfF\bfV=p$, and $M$ is invariant under both $\bfF$ and $\bfV$, the valuation of an eigenvalue of $\bfF$ cannot be greater than $1$.

We can find $M^{ord}, M^{non-ord}$ as follows: Let $f(x)=(x\cdot 1-\bfF|M) \in \Zp[x]$. Factorize $f(x)=f_1(x)f_2(x)$ so that all the roots of $f_1$ have valuation $1$, and all the roots of $f_2$ have valuation less than $1$.  It is clear that we can choose $f_1$ and $f_2$ such that they are monic and $f_1,f_2 \in \Zp[x]$.

Then, we can choose modules $M^{ord} =f_2(\bfF) M$, and $M^{non-ord} = f_1(\bfF)M$ which are clearly invariant under $\bfF$.

Since $\bfF$ and $\bfV$ are commutative (in fact, $\bfF\bfV=\bfV\bfF=p$), both $M^{ord}$ and $M^{non-ord}$ are invariant under $\bfV$ as well. 
\end{proof}

By scaling if necessary, we can assume $M \subset M^{ord} \times M^{non-ord}$.

\begin{notation}

Let $d$ be the dimension of $G$, and choose $l_1,\ldots, l_d \in L$ so that $L$ is generated by them over $\OO_{K'}$. We choose them in such a way that we can write

\begin{eqnarray*} 
l_1&=& m_1'+m_1''		\\
l_2&=& m_2'+m_2''		\\
\vdots\\
l_{d_1}&=&	m_{d_1}'+m_{d_1}''		
\end{eqnarray*}
\begin{eqnarray*} 
l_{d_1+1}&=&		0+m_{d_1+1}''		\\
\vdots\\
l_d&=&		0+m_d''
\end{eqnarray*}
for $m_1',\ldots,m_{d_1}' \in M^{ord}_{\OO_{K'}}$ and $m_1'',\ldots, m_d'' \in M^{non-ord}_{\OO_{K'}}$.

We also let $d_2=d-d_1$.

\end{notation}

\begin{definition}
Let $L^{ord}=\OO_{K'}(m'_1,\ldots, m'_{d_1}) \subset M^{ord}_{\OO_{K'}}$.
Recall the maximal ideal $\mm'$ of $\OO_{K'}$. We define $G^{ord}$ as follows: For an $\OO_{K'}$-algebra $R$, $G^{ord}(R)$ is the subgroup of $G_{M,L}(R)$ given by the pull-back of the fiber product

\[
\begin{array}{ccc}
\Hom(M^{ord} , CW_k (R/ \mm' R)) & \longrightarrow&	\Hom(L^{ord}, (R\otimes \Qp)/P'(R))	\\
&& \uparrow\\
&&\Hom(L^{ord}, R\otimes \Qp)
\end{array}.
\]
 
Here, the pull-back is given by $M \to M^{ord}\times M^{non-ord} \to M^{ord}$ and $L \to L^{ord}$ (the latter being given by $l_1\mapsto m_1', l_2\mapsto m_2', \ldots, l_{d_1}\mapsto m_{d_1}', l_{d_1+1}\mapsto 0, \ldots, l_d\mapsto 0$).

\end{definition}

\end{subsection}


\begin{subsection}{}			\Label{Non-ordinary part}

In this section, we let $K$ be an unramified (finite) extension of $\Qp$, and $k$ be the field of residues of $\OO_K$.

 Let $P(x)=b_dx^d+b_{d-1}x^{d-1}+\cdots+b_0 \in \Zp[x]$, and suppose $b_d$ is a unit, and all the zeros of $P(x)$ have valuation greater than $0$ and less than $1$. (In fact, $b_d$ will be always $1$, but we keep $b_d$ for readability.) Consequently, $b_i/b_d \in p\Zp$ ($i=0,\ldots, d-1$), and $p^ib_i/b_0 \in p\Zp$ ($i=1,\ldots, d$).

Suppose 

$$f(X)=X^p+\alpha_{p-1}X^{p-1}+\cdots + \alpha_1X \in \OO_K[X]$$
satisfies

\[ p|\alpha_i \text{ for }i=1,\ldots,p-1, \text{ and }  v_p(\alpha_1)=1. \]

Let $j(x)=P(x)/b_0-1=\displaystyle \frac{b_d}{b_0}x^d+\frac{b_{d-1}}{b_0}x^{d-1}+\cdots+\frac{b_1}{b_0}x$.

Later in this section, we will define $j(\varphi)$ which imitates the properties of $j(\bfF)$. But first, we make the following formal definition.

\vspace{1.5mm}

\begin{definition}  \Label{Bonn}

\begin{enumerate}
    \item
Recall that $b_d=1$ and $\bfF\bfV=\bfV\bfF=p$. We define
\begin{eqnarray*}
j(\bfF)^{-1}    &=& \left[ \displaystyle \frac{b_d}{b_0} \bfF^d \left( 1+ \frac{b_{d-1}}{b_d} \bfF^{-1}+\cdots + \frac{b_1}{b_d} \bfF^{d-1} \right)  \right]^{-1} \\
&=&  b_0 \displaystyle \frac{\bfV^d}{p^d} \left( 1+ b_{d-1} \frac{\bfV}p + \cdots +b_1 \frac{\bfV^{d-1}}{p^{d-1}} \right)^{-1}
\end{eqnarray*}
where the last line is formally expanded by the Taylor series $(1+x)^{-1}=1-x+x^2- \cdots$.

    \item
    Let $\sigma$ be the $p$-th Frobenius map on $K$ (i.e., $\sigma(x)=x^p \pmod p$ for $x \in \OO_K$).
    Recall that $\bfV x = p x^{\sigma^{-1}}$ for every $x \in K$. For each $n \in \Z$, we define
    
\[ \epsilon^{\sigma^n} = \displaystyle \left( -j(\bfF)^{-1} + j(\bfF)^{-2} - \cdots \right) \cdot \left(-\frac{\alpha_{p-1}^{\sigma^n}}p \right). \]
\end{enumerate}

\end{definition}

\vspace{2mm}
This definition of $\epsilon^{\sigma^n}$ is somewhat similar to $\lambda_n$ in \cite{Kim-Parity}~p.54 and (2.1) in \cite{Kim-Canadian} in the sense that they are all defined by infinite series of similar flavor (although naturally the definition in this paper is much more descriptive and general). However, we believe the definition in this space has much more explanatory power because it can explain the existence of such a constant fully, and also it seems to be defined more naturally.
\vspace{2mm}

\begin{proposition}     \Label{KoKo}
\[ \epsilon^{\sigma^n}+ \displaystyle \frac{b_1}{b_0} \epsilon^{\sigma^{n+1}}+ \cdots + \frac{b_d}{b_0} \epsilon^{\sigma^{n+d}} = \frac{\alpha_{p-1}^{\sigma^n}}p. \]
\end{proposition}

\begin{proof}
We observe

\begin{eqnarray*}
\epsilon^{\sigma^n} &=& j(\bfF)^{-1} \displaystyle \frac{\alpha_{p-1}^{\sigma^n}}p -j(\bfF)^{-1} \left( -j(\bfF)^{-1}+ j(\bfF)^{-2}- \cdots \right) \cdot \left( -\frac{\alpha_{p-1}^{\sigma^n}}p  \right) \\
&=& j(\bfF)^{-1} \displaystyle \frac{\alpha_{p-1}^{\sigma^n}}p -j(\bfF)^{-1} \epsilon^{\sigma^n}     .
\end{eqnarray*}
Thus,

\[ j(\bfF) \epsilon^{\sigma^n}= \displaystyle \frac{\alpha_{p-1}^{\sigma^n}}p - \epsilon^{\sigma^n}, \]
thus

\begin{eqnarray*}
\displaystyle \frac{\alpha_{p-1}^{\sigma^n}}p &=&   \left( 1+j(\bfF) \right) \epsilon^{\sigma^n}      \\
&=&\left( 1+ \displaystyle \frac{b_1}{b_0} \bfF+ \cdots \frac{b_d}{b_0} \bfF^d \right) \epsilon^{\sigma^n}  ,
\end{eqnarray*}
and since $\bfF^i \epsilon^{\sigma^n}= \epsilon^{\sigma^{n+i}}$, our claim follows.
\end{proof}

\begin{definition}
We let $\mathcal P_K$ be the set of power series $\sum_{n=0}^{\infty} A_n x^n$ so that $A_0$ and  $nA_n \in \OO_K$ for $n=1,2,\ldots$.
\end{definition}

\vspace{2mm}

 Recall that $\sigma$ is the $p$-th Frobenius on $K$. Let $\varphi_{f}$ be an operator on $\mathcal P_K$ given by

\[ \varphi_{f} \circ a=\sigma(a) \text{ for }a \in k,\quad \varphi_{f} \circ X=f(X).\]
Recall that $\mathcal P_K / p \OO_K [[x]] \cong \hat{CW}(k[[X]])$. It is easy to see that $\varphi_f$ on $\mathcal P_K / p \OO_K [[x]]$ is equivalent to $\bfF$. Similar to $\log_{\mathcal F_{ss}}(x)$ in \cite{Kobayashi}~p.15 (and also similar to $l(x)$ in \cite{Kim} in its precise form), we define:

\begin{definition}		\Label{Jack and the Beanstalk} We define $l(x)$ by
\[ l(x)=\left[ 1-j(\varphi_f)+j(\varphi_f)^2-\cdots \right] \circ x. \]
\end{definition}

\vspace{2mm}

Readers can see that $\epsilon^{\sigma^n}$ is formally defined as the expansion of $l(x)$ to the minus direction, which may explain its somewhat mysterious properties.

\vspace{2mm}

\begin{proposition}
$l(x)$ is well-defined (i.e., a convergent power series).
\end{proposition}

\begin{proof} Recall that $p^ib_i/b_0 \in p\Zp$ for $i=1,\ldots, d$. The rest is clear.
\end{proof}

\begin{notation}
\begin{enumerate}[(a)]
\item Let $\pi_0=0$, and $\pi_n$ for $n\geq 1$ be  non-zero so that
$$ f^{\sigma^{-n}}(\pi_n)=\pi_{n-1} \text{ for } n=1,2,\ldots.$$
\item Let $\Tr_{n/m}$ denote $\Tr_{K(\pi_n)/K(\pi_m)}$.

\item Let $f^{(i)}(x)$ denote $f^{\sigma^{i-1}}\circ \cdots \circ f^{\sigma}\circ f(x)$.
\end{enumerate}

\end{notation}

\begin{proposition}				\Label{Can}
For any $i$ with $0\leq i<d$, we have

\begin{multline*}b_0 \Tr_{n/n-d} \left(\epsilon^{\sigma^{-n+i}} + (\varphi_{f^{\sigma^{-n}}}^i \circ l^{\sigma^{-n}}) (\pi_n)\right) \\
+ pb_1 \Tr_{n-1/n-d} \left(\epsilon^{\sigma^{-n+1+i}} + (\varphi_{f^{\sigma^{-n+1}}}^i \circ l^{\sigma^{-n+1}}) (\pi_{n-1})\right)		\\%
+ \cdots p^d b_d \left(\epsilon^{\sigma^{-n+d+i}} + (\varphi_{f^{\sigma^{-n+d}}}^i \circ l^{\sigma^{-n+d}}) (\pi_{n-d})\right)=0.		\quad \quad \quad
\end{multline*}
\end{proposition}

\begin{proof}
First, we note

\[(\varphi_{f^{\sigma^{-n+j}}}^i \circ l^{\sigma^{-n+j}})(\pi_{n-j})=l^{\sigma^{-n+i+j}}\left(f^{(i),\sigma^{-n+j}}(\pi_{n-j})\right)=l^{\sigma^{-n+i+j}}(\pi_{n-i-j}).\]

 Then, we note

\begin{eqnarray*} &&\Tr_{n/n-d} l^{\sigma^{-n+i}}(\pi_{n-i}) \\
&=& \Tr_{n/n-d}  \left. \left( \left[ 1-j(\varphi)+j(\varphi)^2-\cdots \right] \circ x \right)^{\sigma^{-n+i}} \right|_{x=\pi_{n-i}} \\
&=&\Tr_{n/n-d} \pi_{n-i} - \Tr_{n/n-d} \left. \left( j(\varphi)\circ  \left[ 1-j(\varphi)+j(\varphi)^2- \right]\circ x \right)^{\sigma^{-n+i}} \right|_{x=\pi_{n-i}} \\
&=& -p^{d-1} \alpha_{p-1}^{\sigma^{-n+i}} -\Tr_{n/n-d} \left[ \displaystyle \frac{b_1}{b_0} l^{\sigma^{-n+i+1}} (f^{\sigma^{-n+i}}(x)) + \displaystyle \frac{b_2}{b_0} l^{\sigma^{-n+i+2}} (f^{(2), \sigma^{-n+i}}(x))	\right.	\\
&&	\quad \quad 
+\cdots+\left. \frac{b_{d}}{b_0} l^{\sigma^{-n+i+d}} (f^{(d), \sigma^{-n+i}}(x))\right]_{x=\pi_{n-i}}	\\
&=& -p^{d-1} \alpha_{p-1}^{\sigma^{-n+i}} - \left[p\displaystyle \frac{b_1}{b_0}	\Tr_{n-1/n-d} l^{\sigma^{-n+i+1}} (\pi_{n-i-1}) 		\right.		\\
&&		 \left. + p^2 \frac{b_2}{b_0} \Tr_{n-2/n-d} l^{\sigma^{-n+i+2}} (\pi_{n-i-2})+\cdots+p^d \frac{b_{d}}{b_0} l^{\sigma^{-n+i+d}} (\pi_{n-i-d}) \right].
\end{eqnarray*}

By Proposition~\ref{KoKo} (replacing $n$ with $-n+i$)
our claim follows.
\end{proof}

\vspace{2mm}

Note the role of $\epsilon^{\sigma^n}$ in the argument.

We note that if $b_0,b_1,\ldots, b_d $ are in an unramified field larger than $\Qp$, then the above argument does not work directly because they are not fixed by $\sigma$, and we may need to find an alternative argument, which is one of the reasons that we assume $P(x)=b_dx^d+\cdots+ b_0 \in \Zp[x]$.

\vspace{1.5mm}

\begin{definition}
\Label{Hear}

\begin{enumerate}
\item
Suppose $M$ is a finitely generated $\Dieudonne$-module, and there is a $\Dieudonne$-homomorphism $M \to M(\bfm) \oplus M^{\bf c}(\bfm)$ for some $\Dieudonne$-modules $M(\bfm)$ and $M^{\bf c}(\bfm)$ so that $M(\bfm)=\Dieudonne \cdot \bfm$ for some $\bfm$. 

\item For any $m \in M$, the minimal polynomial of $m$ over $\Qp$ is the monic polynomial $P(x) \in \Qp[x]$ with the smallest degree such that $P(\bfF) m=0$.

\item

We assume the minimal polynomial of $\bfm$ over $\Qp$ is $P(x)$.

\end{enumerate}
\end{definition}

\vspace{1.5mm}

Assuming $\bfF$ is a topological nilpotent on $M(\bfm)$ (i.e., $\bfF^n \to 0$ as action on $M(\bfm)$ as $n\to0$), it is clear that $\Zp[\bfF]\bfm$ is a subgroup of finite index in  $M(\bfm)$.

\vspace{1.5mm}

\begin{remark}
If $\Dieudonne_k=W(k) [\bfF, \bfV]$ for $k \not= \Z/p\Z$ so that $W(k)$ is strictly bigger than $\Zp$, then the above definition of a minimal polynomial may not make sense. For example, suppose that for some $p(x) \in W(k)[x]$, $p(\bfF)\bfm=0$. Then, for some $a \in W(k)$, $p(\bfF) (a\bfm)$ may not be $0$. This is another reason that we assume the group scheme is defined over a totally ramified local field (thus $k=\Z/p\Z$).
\end{remark}
\vspace{1.5mm}

Recall that $K'$ is a totally ramified extension of $\Qp$. (Thus, $K'$ is linearly disjoint over $\Qp$ from $K$.) Recall $L$ is a (free) $\OO_{K'}$-submodule of $M_{\OO_{K'}}$.

\vspace{1.5mm}

\begin{definition}			\Label{St. Petersburg}
Assume $P(x)\in \Zp[X]$, and all its roots have valuation less than 1, and all its coefficients except the leading coefficient are in $p\Zp$.

\begin{enumerate}[(a)]
\item Recall the polynomial $f(x)$. Define $l(x)$ associated to $P(x)$ and $f(x)$ as in Definition~\ref{Jack and the Beanstalk}.

\item Recall $\overline {\mathcal P}_K=\mathcal P_K/ p\OO_K[[X]]$. 

Fix $C\in \Z$ ($C>0$) so that $C$ annihilates $M(\bfm)/\Zp[\bfF]\bfm$. For each $n \geq \Z$, define $\bfx^{\sigma^n} \in \Hom_{\Zp[\bfF]} (M(\bfm) \oplus M^{\bf c}(\bfm), \overline{\mathcal P})$ by

\begin{eqnarray*}	\bfx^{\sigma^n} : M(\bfm) \oplus M^{\bf c}(\bfm) 		&\to 		&    \overline{\mathcal P}		\\
\bfm		&\mapsto		& C \cdot l^{\sigma^n} (x)		\\
M^{\bf c}(\bfm) 	&\mapsto		&0
\end{eqnarray*}
and expand linearly. (Expand $\Zp[\bfF]$-linearly to $\Zp[\bfF]\bfm$, and then to the entire $M(\bfm)$ by scaling, which is possible because $\Zp[\bfF]\bfm$ is a subgroup of $M(\bfm)$ of finite index.)

\item Also, for each $n \in \Z$, define $\tilde \bfx^{\sigma^n}  \in \Hom_{\Zp}(M(\bfm) \oplus M^{\bf c}(\bfm), \mathcal P_K)$ by
\[ \tilde \bfx^{\sigma^n} (\bfF^k \bfm)=C \cdot \left(	\epsilon^{\sigma^{n+k}}+\varphi_{f^{\sigma^n}}^k \circ l^{\sigma^n} (X)		\right), \quad \text{ for } k=0,1,\ldots,d-1\] and 
\[ \tilde \bfx^{\sigma^n}:  M^{\bf c}(\bfm) \mapsto 0,		\]
and extend it linearly.

\end{enumerate}
\end{definition}

\vspace{2mm}

Clearly, $\tilde \bfx^{\sigma^n}$ modulo $p\OO_K[[x]]$ is $\bfx^{\sigma^n}$.

Similar to \cite{Kim}~Notation~4.8, we define the following.

\vspace{2mm}

\begin{definition}			\Label{Wizard of Oz}
Choose generators $\bfl_1, \ldots, \bfl_d$ of $L$. Recall that $L\subset M_{\OO_{K'}}$. Via $M \to M(\bfm) \oplus M^{\bf c}(\bfm)$, consider $L$ as a submodule of $(M(\bfm) \oplus M^{\bf c}(\bfm))_{\OO_{K'}}$. For each $h=1,\ldots, d$, write

\[ \bfl_h=(\bfl_{ij}^{(h)})_{(i,j)\in I_0}  \in \varprojlim_{(i,j) \in I_0} \mm^i \otimes M^{(j)}  ,	\]
\[ \bfl_{ij}^{(h)}=\left(\sum_{k=0}^{d-1} \alpha_{k, ij}^{(h)} \bfF^k \bfm , m_2\right) \in (\mm^i \otimes M(\bfm)^{(j)}) \oplus (\mm^i \otimes M^{\bf c} (\bfm)^{(j)}), \]
where $\sum_{k=0}^{d-1} \alpha_{k, ij}^{(h)} \bfF^k \bfm \in \mm^i \otimes M(\bfm)^{(j)}$, and $m_2 \in \mm^i \otimes M^{\bf c} (\bfm)^{(j)}$.

For each $h=1,2,\ldots, d$ and each $(i,j) \in I_0$, we can write

\[	\bfF^j \left( \sum_{k=0}^{d-1} \alpha_{k, ij}^{(h)} \bfF^k \bfm \right)	=	\sum_{k=0}^{d-1} \beta_{k, ij}^{(h)} \bfF^k \bfm\]
for some $\beta_k^{(ij)}\in C^{-1} \mm^i$.
For each $n\in \Z$, we define $\bfy^{\sigma^n} \in \Hom_{\OO_{K'}} (L, K'[[x]])$ by

\begin{eqnarray*}
\bfy^{\sigma^n} (\bfl_h)		&=&	C \cdot 	\left( \sum_{(i,j)\in I_0} \sum_{k=0}^{d-1} \beta_{k, ij}^{(h)} \tilde\bfx^{\sigma^n}  (\bfF^k \bfm)	\right) 	\\
&=&  C \cdot 	\left( \sum_{(i,j)\in I_0} \sum_{k=0}^{d-1} \beta_{k, ij}^{(h)} (\epsilon^{\sigma^{n+k}} +\varphi_{f^{\sigma^n}}^k \circ l^{\sigma^n}(X)) \right)
\end{eqnarray*}
for each $h =1,\ldots, d$.
\end{definition}

\vspace{2mm}


\begin{definition}			\Label{Battle of Canae}
Recall Fontaine's map $\omega$ and group $P'$ which are defined as follows (\cite{Fontaine}~Chapter~2): $\omega$ is given by

\begin{eqnarray*}
\omega: \hat{CW}(g) &\to& \Qp \otimes g     \\
(\cdots, a_{-n}, \cdots, a_{-1}, a_0)   &\mapsto&   \sum_{n=0}^{\infty} p^{-n} a_{-n}^{p^n}.
\end{eqnarray*}

And, for any $\OO_{K'}$-algebra $g$, $P'(g)$ is an $\OO_{K'}$-submodule of $\Qp \otimes g$ generated by $p^{-n} a^{p^n}$ for all $n \geq 0$ and all $a \in \mm' \cdot g$ (you may also see \cite{Kim}~Definition~3.1 for a reference).

\begin{enumerate}[(a)]
\item Let $G_{M(\bfm) \oplus M^{\bf c} (\bfm), L}(\OO_{K\cdot K'}[[x]])$ be the fiber product given by the following diagram

\begin{eqnarray*}	\Hom_{\Zp[\bfF]}(M(\bfm) \oplus M^{\bf c} (\bfm), \hat{CW}(k[[x]])) 		&\stackrel{\omega}\to&	\Hom_{\OO_{K'}}(L, K\cdot K' [[x]]/P'(\OO_{K\cdot K'}[[x]]))		\\
&&	\uparrow		\\
&& \Hom_{\OO_{K'}}(L, K\cdot K'[[x]]).
\end{eqnarray*}

\item For each $n \in \Z$, let $P^{\sigma^n} =(\bfx^{\sigma^n} , \bfy^{\sigma^n} ) \in G_{M(\bfm) \oplus M^{\bf c} (\bfm), L}(\OO_{K\cdot K'} [[x]])$. It is clear that $P^{\sigma^n}$ is well-defined.

\item Note that there is a natural pull-back $\iota: \Hom_{\Zp[\bfF]}(M(\bfm) \oplus M^{\bf c} (\bfm), \hat{CW}(k[[x]])) \to \Hom_{\Zp[\bfF]}(M, \hat{CW}(k[[x]]))$. Also, by Perrin-Riou's lemma (see the discussion \cite{Perrin-Riou}~Section~3.1 right before Theorem~3.1), we have
\[ \Hom_{\Zp[\bfF]}(M, \hat{CW}( k[[x]])) \cong \Hom_{\Dieudonne}(M, \hat{CW}( k[[x]])). \]

So, via the pull-back $\iota$, there is a map 
\[G_{M(\bfm) \oplus M^{\bf c} (\bfm), L} (\OO_{K \cdot K'}[[x]]) \to G_{M, L} (\OO_{K\cdot K'}[[x]]).\]

\item By abuse of notation, let $P^{\sigma^n}$ denote the image of $P^{\sigma^n}$ in $G_{M, L} (\OO_{K\cdot K'}[[x]])$ under this pull-back map.

\item Where $Q \in G_{M, L}(\OO_{K\cdot K'}[[x]])$ and $\pi$ is an element in some local field with positive valuation, let $Q(\pi) \in G_{M, L}( \OO_{K\cdot K'}[\pi])$ denote $Q$ with $\pi$ substituted for $x$.

Then, clealry $P^{\sigma^n}(\pi_n) \in G_{M, L}(\OO_{K\cdot K'}[\pi_n])$.
\end{enumerate}

\end{definition}

\vspace{2mm}

Then, finally we obtain the following. (See \cite{Kim}~Proposition~4.10 for comparison.)

\vspace{2mm}

\begin{proposition}			\Label{Mexico}

Recall $P(x)=b_dx^d+\cdots+b_0$ to which the construction of $P^{\sigma^n}$ is associated.

For each $n\geq d$, we have that modulo the torsions of $G_{M,L}(\OO_{K\cdot K'}[\pi_n])$,
\[ \sum_{s=0}^d p^s b_s \Tr_{n-s/n-d} P^{\sigma^{-n+s}}(\pi_{n-s})=0.\]


\end{proposition}

\begin{proof}
 Since the equality is modulo torsions, we only need to check the ``$\bfy$'' part. Then, the claim follows from Proposition~\ref{Can}.

\end{proof}

\end{subsection}


\begin{subsection}{}			\Label{PPP}
Recall that $K$ is a (finite) unramified extension of $\Qp$, and $K'$ is a totally ramified extension of $\Qp$. By the method in the previous section (Section~\ref{Non-ordinary part}), we can construct a series of local points satisfying a certain norm relation, and doing so does not require any condition on $K'$ other than it being totally ramified over $\Qp$.

However, we need to generate a sufficient number of series of local points, and we need to do so for a group scheme with a mix of ordinary reduction types and non-ordinary reduction types. In this subsection, we explain how to generate enough points, and to do so, we need to assume $K'$ is generated by $p^N$-torsions of a certain Lubin-Tate group as below.

\begin{definition} 		\Label{Gal}

Fix a uniformizer $\rho$ of $K$.

Choose $\zeta \in (\OO_K^{\times})_{tor}$ so that $\Zp[\zeta]=\OO_K$. For each $i=0,1,\ldots, [K:\Qp]-1$, we choose
$$ f_i(x)= x^p+\alpha_{i, p-1}x^{p-1}+\cdots+ \alpha_{i, 1}x$$
so that $\alpha_{i, 1},\ldots, \alpha_{i, p-1} \in p\OO_K$, and
$$ \alpha_{i, p-1}=\zeta^i p,\quad \alpha_{i, 1}=\rho. $$

For each $i=0,1,\ldots, [K:\Qp]-1$, we let $\pi_{i, 0}=0$, and choose non-zero $\pi_{i, n}$ for each $n\geq 1$ so that
\[ f^{\sigma^{-n}}_i(\pi_{i, n})=\pi_{i, n-1}. \]

By the local class field theory, $K(\pi_{i,n})$ does not depend on $i$. We let $K(\pi_n)$ denote any $K(\pi_{i,n})$, and $K(\pi_{\infty})$ denote $\cup_n K(\pi_n)$.

\end{definition}

\vspace{2mm}

\begin{assumption}

We suppose $K'\cdot K=K(\pi_N)$ for some $N$.

\end{assumption}

For example, we can consider the case $K'=\Qp(\pi_{p^N})$ for some $N$, and $K$ is the unique unramified quadratic extension of $\Qp$. One thing to note is that contrary to its appearance, $K'$ is not an extension of $K$ unless $K=\Qp$ because $K'$ is totally ramified over $\Qp$, and $K$ is unramified over $\Qp$. 

\vspace{2mm}

Recall that \textbf{minimal polynomial} of $\bfm \in M$ is the (non-zero) monic polynomial $p(x) \in \Qp[x]$ of minimal degree so that $p(\bfF) \bfm=0$.


\begin{notation} Recall that $G$ is a (smooth) formal group scheme over $\OO_{K'}$ of dimension $d$, and $M$ is the associated Dieudonne module. From Section~\ref{Ordinary part}, recall the definitions of $M^{ord}$ and $M^{non-ord}$. As in that section, let $d_1$ and $d_2$ be the minimum numbers of $\Qp[\bfF]$-generators of $M^{ord} \otimes \Qp$ and $M^{non-ord} \otimes \Qp$, respectively. (Then, $d=d_1+d_2$.)

\begin{enumerate}[(a)]
\item Choose generators $\tilde\bfm_1,\ldots, \tilde \bfm_{d_2}$ of $M^{non-ord}$, and by multiplying some polynomials (with coefficients in $\Qp$) of $\bfF$ to them if necessary, obtain $\bfm_1, \ldots, \bfm_{d_2}$ so that their minimal polynomials are irreducible over $\Qp$.

\item
Let $p_i(x)\in \Zp[x]$ $(i=1,\ldots, D)$ be the minimal polynomial of $\bfm_i$. We write

\[ p_i(x)=b^{(i)}_{d_i} x^{d_i}+b_{d_i-1}^{(i)}x^{d_i-1}+\cdots + b_0^{(i)}. \]
(In fact, $b_{d_i}^{(i)}=1$ by definition, but for simplicity, we keep $b^{(i)}_{d_i}$.) By the definition of $M^{non-ord}$, $b_0^{(i)},\ldots, b_{d_i-1}^{(i)} \in p\Zp$.

\end{enumerate}
\end{notation}

\vspace{2mm}

By the definition of $M^{non-ord}$, it is clear that all the roots of $p_i(x)$ have valuation less than $1$.


\begin{definition}		\Label{Leipzig}
For $i=0,1,\ldots, [K:\Qp]-1$ and $j=1,2,\ldots, d$, we define the following: For each $i$, and every $n\in \Z$, we define $\varphi_{f_i^{\sigma^n}}$ acting on $\mathcal P_K$ by

$$\varphi_{f_i^{\sigma^n}} \circ a=\sigma(a) \text{ ~~for } a \in k, \text{ and } \varphi_{f_i^{\sigma^n}} \circ X=f_i^{\sigma^n}(X).$$

For each $j=1,\ldots, d_2$, write $J_j(x)=p_j(x)/b^{(j)}_{0}-1$. As in Definition~\ref{Jack and the Beanstalk}, we define

\[ l_{f_i,p_j}(x)=\left[ 1-J_j(\varphi_{f_i})+J_j(\varphi_{f_i})^2 - \cdots \right] \circ x. \]

Also, for each $k \in \Z$ we define $\epsilon^{\sigma^k}_{f_i, p_j}$ attached to $f_i, p_j$ by the method of Definition~\ref{Bonn}.

\vspace{1.5mm}

\begin{enumerate}[(a)]
\item
Let $M(\bfm_j)=\Dieudonne \cdot \bfm_j$, and $M^{\bf c}(\bfm_j)=\prod_{i\not=j} \Dieudonne \cdot \bfm_i \times M^{ord}$. (See Notation~\ref{Hear}.) Then, there is a canonically defined $\Dieudonne$-homomorphism $M \to M(\bfm_j) \oplus M^{\bf c} (\bfm_j)$.

\vspace{1.5mm}

\item
As in Definition~\ref{St. Petersburg}, fix a sufficiently large integer $c>0$ so that $c$ annihilates $M(\bfm_j)/\Zp[\bfF]\cdot \bfm_j$ for every $j$, and for each $n \in \Z$, define $\bfx_{f_i,p_j}^{\sigma^n} \in \Hom_{\Zp[\bfF]} (M, \overline{\mathcal P})$ by

\begin{eqnarray*}	\bfx_{f_i,p_j}^{\sigma^n} : M 		&\to 		&    \overline{\mathcal P}		\\
\bfm_j	&\mapsto		&c \cdot l_{f_i, p_j}^{\sigma^n} (x)		\\
m (\in M^{\bf c}(\bfm_j)) 	&\mapsto		&0
\end{eqnarray*}
and extend it $\Zp[\bfF]$-linearly (which is possible because $c$ annihilates $M(\bfm_j)/\Zp[\bfF] \cdot \bfm_j$).

Also, for each $n \in \Z$, define $\tilde \bfx_{f_i,p_j}^{\sigma^n}  \in \Hom_{\Zp}(M, \mathcal P_K)$:

\[ \tilde \bfx_{f_i,p_j}^{\sigma^n} (\bfF^k \bfm_j)= c \cdot \left( \epsilon^{\sigma^k}_{f_i, p_j} +\varphi_{f_i^{\sigma^n}}^k \circ l_{f_i, p_j}^{\sigma^n} (X)	\right), \text{ for } k=0,1,\ldots,d_j-1,\] and 
\[ \tilde \bfx_{f_i,p_j}^{\sigma^n} (m)=0 \text{ for } m \in M^{\bf c}(\bfm_j).\]

Then, define $\bfy_{f_i, p_j}^{\sigma^n}$ as in Definition~\ref{Wizard of Oz} using $\tilde \bfx_{f_i,p_j}^{\sigma^n}$.

\vspace{1.5mm}

\item Define $P_{f_i,p_j}^{\sigma^n} = ( \bfx^{\sigma^n}_{f_i, p_j}, \bfy^{\sigma^n}_{f_i, p_j}) \in G_{L, M(\bfm_j) \times M^{\bf c}(\bfm_j)} (\OO_{K\cdot K'} [[x]])$.
(See Definition~\ref{Battle of Canae} for the definition of $ G_{L, M(\bfm_j) \times M^{\bf c}(\bfm_j)} (\OO_{K\cdot K'} [[x]])$.)
\end{enumerate}
\end{definition}

\vspace{2mm}

Via $M \to M(\bfm_j) \oplus M^{\bf c} (\bfm_j)$, there is a natural homomorphism $G_{L,  M(\bfm_j) \oplus M^{\bf c}(\bfm_j)} \to G_{L, M}$. By combining with the isogeny $G_{L, M} \to G$, we obtain a homomorphism $\varrho_j: G_{L,  M(\bfm_j) \oplus M^{\bf c}(\bfm_j)} \to G$.

\vspace{2mm}

\begin{definition}			\Label{Vio}
For each $f_i$ for $i=0,1,\ldots, [K:\Qp]-1$, $p_j$ 
for $j=1,\ldots, d_2$, and $n \in \Z$, we let
\[ Q_{f_i,p_j}^{non-ord, \sigma^{-(N+n)}}(\pi_{i,N+n})=\varrho_j (P_{f_i, p_j}^{\sigma^{-(N+n)}}(\pi_{i, N+n})) \in G(\OO_K [\pi_{N+n}]). \]
\end{definition}

\vspace{2mm}

Clearly, $Q_{f_i,p_j}^{non-ord, \sigma^{-(N+n)}}(\pi_{i, N+n})$ should satisfy the relation in Proposition~\ref{Can}:

\begin{multline*} b_0^{(i)} \Tr_{N+n/N+n-d_i} Q_{f_i,p_j}^{non-ord, \sigma^{-(N+n)}}(\pi_{i, N+n})		\\
+p b_1^{(i)} \Tr_{N+n-1/N+n-d_i} Q_{f_i,p_j}^{non-ord, \sigma^{-(N+n-1)}}(\pi_{i, N+n-1})		\\
+ \cdots +p^{d_i} b_{d_i}^{(i)} Q_{f_i,p_j}^{non-ord, \sigma^{-(N+n-d_i)}}(\pi_{i, N+n-d_i})=0
\end{multline*}
modulo torsions.

On the other hand, since $G^{ord}$ is of multiplicative type with dimension $d_1$, $\varprojlim_n G^{ord}(\OO_K[\pi_{N+n}])$ has rank $d_1 \cdot [K\cdot K':\Qp]$. In other words, we have the following:
\vspace{2mm}

\begin{proposition}			\Label{Berlin}
We can choose a set of elements $\{(Q_{i, N+n}^{ord})_n\}_{i=1,2,\ldots, d_1 \cdot [K\cdot K':\Qp]}$ of $\varprojlim_n G^{ord}(\OO_{K}[\pi_{N+n}])$ so that it generates a free $\Zp[[\Gal(K(\pi_{\infty})/K(\pi_N))]]$-submodule of rank $d_1\cdot [K\cdot K':\Qp]$. In particular, we can choose them so that $\{Q_{i, N+n}^{ord} \}_{i=1,2,\ldots, d_1 \cdot [K\cdot K':\Qp]}$ generates a subgroup of $G^{ord}(\OO_{K}[\pi_{N+n}])$ whose quotient has a bounded rank as $n$ varies.

\end{proposition}
\end{subsection}

\end{section}


\begin{section}{Perrin-Riou characteristics}

\Label{Turin}

The construction of local points in the previous section can be applied to the following setting:

\begin{notation}			\Label{300}
To simplify our assumptions, let $F$ be a number field, $F_{\infty}$ be its extension so that $\Gal(F_{\infty}/F) \cong \Zp$, and $H$ be a number field so that every prime of it above $p$ is unramified over $H/\Q$. For any prime $\mfq$ of $H\cdot F_{\infty}$ above $p$, let $\mfq$ also denote $\mfq \cap H$ by abuse of notation, and let $\mfp$ denote both $\mfq \cap F_{\infty}$, and $\mfq \cap F$ again by abuse of notation.
\end{notation}

\begin{assumption}

For every prime $\mfq$ of $HF_{\infty}$ above $p$, $(HF_{\infty})_{\mfq}=H_{\mfq} (\pi_{\infty})$ for the $p$-power torsions $\pi_n$ of some Lubin-Tate group of height 1 defined over $\OO_{H_{\mfq}}$ (equivalently, $(HF_{\infty})_{\mfq}$ is abelian over $\Qp$, and $\Gal((HF_{\infty})_{\mfq}/H_{\mfq}) \cong \Z_p^*$), and $(HF)_{\mfq}=H_{\mfq} (\pi_{N_{\mfq}})$ for some $N_{\mfp}$.
\end{assumption}

There are some obvious cases where this assumption holds true: For example, $F=\Q(\zeta_{p^N})$ for some $N>0$, $F_{\infty}=\Q(\zeta_{p^{\infty}})$, and $H$ is a number field such that every prime of $H$ above $p$ is unramified over $H/\Q$.

\begin{notation}
Let $A$ be an abelian variety over $F$ which has good reduction at every prime of $F$ above $p$. Let $A^{\vee}/F$ be its dual abelian variety.
\end{notation}

\begin{remark}
It is probably difficult to determine whether a given abelian variety has good reduction at every prime or not. We believe that the result of this paper holds regardless of that. However, to the best of our knowledge, there is not a proper theory for abelian varieties which have bad (and non-multiplicative--in other words, unstable) reduction.
\end{remark}


\begin{notation}			\Label{AFG}

Let

\[ \Gamma=\Gal(H\cdot F_{\infty}/ H\cdot F)\cong \Gal(F_{\infty}/F). \]
\end{notation}

Let $\mathcal A_{\mfp}$ be the Neron model of $A$ over $\OO_{F_{\mfp}}$, and let $G_{\mfp}=\varinjlim_n \mathcal A_{\mfp}[p^n]$ (as the injective limit of finite group schemes). Let $M_{\mfp}$ be its Dieudonne module.

As in Proposition~\ref{Optical Fiber}, $M_{\mfp}\otimes \Qp$ has decomposition $\Big(M_{\mfp}^{ord}\otimes \Qp\Big) \times\Big(M_{\mfp}^{non-ord} \otimes \Qp\Big)$, and we may let $d^{ord}_\mfp$ be the number of generators of $M_{\mfp}^{ord}$, and $d^{non-ord}_\mfp$ be the number of generators of $M_{\mfp}^{non-ord}$.

Then, as in Definitions~\ref{Leipzig} and \ref{Vio}, there are

\[ Q_{f_i, p_j}^{non-ord, \sigma^{-n}}(\pi_{i,N_{\mfq}+n}) \in G(\OO_{H_{\mfq}F_{\mfp}}(\pi_{N_{\mfq}+n})) \]
for some $f_i$'s for $i=0,1,\ldots, [H_{\mfq}:\Qp]-1$ and $p_j$'s for $j=1,\ldots, d^{non-ord}_{\mfp}$, which we will not specify. Fix $\tilde \tau_1, \ldots, \tilde \tau_{[F_{\mfp}:\Qp]} \in \Gal((HF_{\infty})_{\mfq}: H_{\mfq})$ which lift all the elements of $\Gal((HF)_{\mfq}/H_{\mfq})$. Then, $Q_{f_i, p_j}^{non-ord, \sigma^{-n}}(\pi_{i,N_{\mfq}+n})^{\tilde \tau_l}$ for all $i,j,l$ give us $d^{non-ord}_{\mfp} \cdot [(HF)_{\mfq}:\Qp]$ points.

And, by Proposition~\ref{Berlin}, we can choose the following points:

\begin{definition} We choose the points $Q_{i,n}^{\mfq, ord} \in G(\OO_{H_\mfq F_{\mfq}[\pi_{N_{\mfq}+n}]})$ for $i=1,\ldots, d^{ord}_{\mfp}\cdot [(H F)_{\mfq}: \Qp]$ so that 

\[N_{(H F)_{\mfq}(\pi_{N_{\mfq}+n+1})/(H F)_{\mfq}(\pi_{N_{\mfq}+n})} Q_{i,n+1}^{\mfq, ord}=Q_{i,n}^{\mfq, ord}.\]

We choose them so that $\{Q_{i,n}^{\mfq, ord}\}_i$ generates a subgroup of $G(\OO_{H_\mfq}[\pi_{N_{\mfq}+n}])$ whose quotient has a bounded rank as $n$ varies.
\end{definition}
\vspace{1.5mm}

Altogether, we have $d\cdot [(HF)_{\mfq}:\Qp]$ sequences of local points, and through all primes $\mfq$ of $H \cdot F$ lying above $p$, we have $d\cdot [H F:\Q]$ sequences of local points. For convenience, we list them $\{ Q_{i,n}\}_{i=1,\ldots, d\cdot [H F:\Q] }$.

Note that for each $i$, $\{ Q_{i,n}\}_n$ satisfies

\begin{eqnarray}	\Label{Takgu}
 \sum_{s=0}^{e_i} g^{(i)}_s \Tr_{n-e_i+s/n-e_i} Q_{i,n-e_i+s}=0
\end{eqnarray}
modulo torsions for some irreducible $q_i(x)=g^{(i)}_{e_i}x^e+\cdots + g^{(i)}_0 \in \Zp[x]$.

The following is a standard definition:

\begin{definition}[Relaxed Selmer groups] \Label{Relaxed Selmer}

Let $\bfA'=\cup_n A^\vee [p^n]$.
For  an extension $L$ of $F$,
\[ \Selr(\bfA'/L)\stackrel{def}= \ker \left( H^1(L, \bfA')\to \prod_{v} \ds \frac{H^1(L_v, \bfA')}{A'(L)\otimes \Qp/\Zp} \right)\]
where $v$ runs over all primes of $L$ not lying above $p$.
\end{definition}

\vspace{2mm}

Clearly, $\Selr$ satisfies the Control Theorem. (See \cite{Mazur},\cite{Greenberg}. They assume that the abelian variety has good ordinary reduction at every prime above $p$, but their arguments can be easily adapted because $\Selr$ is a relaxed Selmer group--in other words, it has no local condition for primes above $p$.) In other words,

\[ \Selr(\bfA'/H \cdot F_n) \to \Selr(\bfA'/ H \cdot F_{\infty})^{\Gal(H \cdot F_{\infty}/ H \cdot F_n)} \]
has finite and bounded kernel and cokernel as $n$ varies.

Set

\[ \Lambda\stackrel{def}=\Zp[[\Gamma]]\cong \Zp[[X]] \]
by choosing a topological generator $\gamma$ of $\Gamma$, and setting $\gamma=X+1$.

\vspace{2mm}


\begin{assumption}		\Label{Danube}
Let $M^{\vee}$ denote $\Hom(M,\rat/\Z)$ (``the Pontryagin dual''). We assume
\begin{eqnarray*}
\rank_{\Lambda} \Selr(\bfA'/ H \cdot F_{\infty})^{\vee} =		 \dim A \cdot   [H \cdot F :\Q].
\end{eqnarray*}

\end{assumption}

\vspace{2mm}

As we can see easily by standard techniques of Iwasawa theory, this assumption is true if $\Sel(\bfA'/H \cdot F_n)^{\chi}$ is finite for any character $\chi$ of $\Gal(H \cdot F_n/ H \cdot F)$ for any $n$.

\vspace{2mm}

\begin{notation}
\begin{enumerate}
\item For each $n\geq 0$, let
\[ \Gamma_n=\Gamma/\Gamma^{p^n}, \quad \Lambda_n=\Zp[\Gamma_n]. \]
\item For a group $M$ on which $\Gamma$ acts, we let
\[ M_{/\Gamma^{p^n}}=M/\{ (1-a)\cdot m \; |\; a \in \Gamma^{p^n}, m \in M\}. \]
Equivalently, where $\gamma$ is a chosen topological generator of $\Gamma$,
\[ M_{/\Gamma^{p^n}}=M/(1-\gamma^{p^n})\cdot M. \]
\end{enumerate}
\end{notation}

\vspace{1.5mm}


\begin{definition}
\begin{enumerate}
\item
Let $s=\dim A \cdot [H \cdot F : \Q]$.

\item Let
\[ S_{tor}=\left( \Selr(\bfA'/ H \cdot F_{\infty})^{\vee} \right)_{\Lambda-torsion}.\]
\end{enumerate}
\end{definition}

\vspace{2mm}

If we assume Assumption~\ref{Danube}, then there is a short exact sequence

\begin{eqnarray}		\Label{Ice-Cream}	 0  \to \Selr(\bfA'/H \cdot F_{\infty})^{\vee}/S_{tor} \to \Lambda^s \to C \to 0
\end{eqnarray}
for a finite group $C$. This induces

\[ \alpha_n': \Selr(\bfA' /H \cdot F_n)^{\vee} \to \Lambda_n^s .\]
We note that there is a map
\[ \beta_n:  \prod_{\mfq} A((HF_n)_{\mfq}) \to \Selr(\bfA'/H \cdot F_n)^{\vee} \]
given by the local Tate duality.

\begin{definition}
\begin{enumerate}[(a)]
\item Let $R_{i, n} \in \Lambda_n^s$ be the image of $Q_{i,n}$ under $\alpha_n' \circ \beta_n$.

\item Let $\Proj_n^m$ be the natural projection from $\Lambda_m$ to $\Lambda_n$ ($m\geq n$).
\end{enumerate}
\end{definition}

Recall that each $\{ Q_{i,n} \}_n$ satisfies (\ref{Takgu}) for some $q_i(x)$.

In the following, $\Lambda_{\alpha}$ denotes the set of power series $f(T) \in \overline{\C}_p[[T]]$ satisfying $|f(x)| < C |1/\alpha^n|$ for some fixed $C>0$ for every $n \geq 1$ and $x \in \C_p$ with $|x| < |1/\sqrt[p^n]p|$.

\begin{proposition}
For each root $\alpha$ of $q_i(x)$, there is $f_{\alpha, i} \in \Lambda_{\alpha}^s$ so that for some fixed constant $c$,

\[ R_{i, n} \equiv \sum_{\alpha} f_{\alpha, i} \alpha^{n+1} \pmod{c^{-1}((T+1)^{p^n}-1) \Lambda} \]
for every $n$.

\end{proposition}
\begin{proof} As in \cite{Kim}~Lemma~4.26, this is due to \cite{Perrin-Riou}~Lemme~5.3.
\end{proof}


\vspace{2mm}


\begin{definition}
\Label{Madrid}
First, we choose a generator $g_{tor}\in \Lambda$ of the characteristic ideal of $(\Selr(\bfA/HF_{\infty})^{\vee})_{\Lambda-torsion}$. Now, recall $q_i(x)$ associated to $\{ Q_{i,n} \}_n$ (see the discussion before Definition~\ref{Relaxed Selmer}). We choose a root $\alpha_i$ of each $q_i(x)$.

Then, we define
\[ \bfL_{\{ \alpha_k \}} \stackrel{def}= g_{tor}\times \det [f_{\alpha_1, 1}, \cdots,  f_{\alpha_s, s}].\]
\end{definition}

\vspace{2mm}

\begin{remark}
Readers may wonder why we do not define $\bfL_{\alpha}$ for a single zero $\alpha$ of $q(x)=\prod_i q_i(x)$ as the second author did in \cite{Kim}. That is because if we choose a single zero $\alpha$ of one specific polynomial $q_i(x)$, and define $\bfL_{\alpha}$ in the manner of \cite{Kim}, then it has the effect of negating all the local points $\{ Q_{j,n} \}_n$ with $q_j\not= q_i$. The next example will illustrate what we mean.

This was not an issue in \cite{Kim} because it assumed that the abelian variety was one-dimensional, and therefore, there was only one polynomial $q_1(x)$.

\end{remark}

\vspace{2mm}

\begin{example}
Suppose we have two sequences of points $\{ Q_n \}_n, \{ Q_n'\}_n$ satisfying $N_{n+1/n}Q_{n+1}=Q_n$, and $N_{n+1/n}Q_{n+1}'=2Q_n'$. Define $R_n, R_n' \in \Zp[\Gamma_n]^2$ as above. Then, their corresponding polynomials (again, as above) are $x-1$ and $x-2$, thus the roots are $1,2$.

Note that $q(x)=(x-1)(x-2)=x^2-3x+2$, and we observe,

\begin{multline*}N_{n+2/n} Q_{n+2}-3N_{n+1/n} Q_{n+1}+2Q_n	\\
= N_{n+1/n}(N_{n+2/n+1} Q_{n+2}-Q_{n+1})-2(N_{n+1/n} Q_{n+1}-Q_n)=0,
\end{multline*}
and similarly

\begin{eqnarray*}N_{n+2/n} Q_{n+2}'-3N_{n+1/n} Q_{n+1}'+2Q_n'	=0,
\end{eqnarray*}
thus

\[ 
\begin{bmatrix} R_{n+2}^t	\\	
R_{n+1}^t		
\end{bmatrix}
= \begin{bmatrix}	3	&	-2\\
1	&	0
\end{bmatrix}
\begin{bmatrix}	R_{n+1}^t	\\
R_n^t
\end{bmatrix} \pmod{ (1+x)^{p^n}-1},
\]

\[  
\begin{bmatrix} {R_{n+2}'	}^t\\	
{R_{n+1}'	}^t	
\end{bmatrix}
= \begin{bmatrix}	3	&	-2\\
1	&	0
\end{bmatrix}
\begin{bmatrix}	{R_{n+1}'}^t	\\
{R_n'}^t
\end{bmatrix} \pmod{ (1+x)^{p^n}-1}.
\]
Since $\begin{bmatrix}	3	&	-2\\		1	&	0	\end{bmatrix}= \begin{bmatrix} 1&2\\ 1&1 \end{bmatrix} \begin{bmatrix} 1&0 \\ 0&2 \end{bmatrix} \begin{bmatrix} 1&2\\ 1&1 \end{bmatrix}^{-1},$ by the method of \cite[Lemme~5.3]{Perrin-Riou}, we have

\begin{eqnarray*} 
\begin{bmatrix} f_1 \\ f_2		\end{bmatrix}
&=&	\varprojlim 
\begin{bmatrix}	1&2\\ 1&1	\end{bmatrix}^{-1} 
\begin{bmatrix}	R_{n+1}^t	\\	R_n^t	\end{bmatrix}		\\
&=&	\varprojlim
\begin{bmatrix}	-1&2	\\	1&-1	\end{bmatrix}
\begin{bmatrix}	R_{n+1}^t	\\	R_n^t	\end{bmatrix}		\\
&=&	\varprojlim
\begin{bmatrix}	-R_{n+1}^t	+2R_n^t	\pmod{ (1+x)^{p^n}-1}	\\		0	\end{bmatrix},
\end{eqnarray*}
and

\begin{eqnarray*} 
\begin{bmatrix} f_1' \\ f_2'		\end{bmatrix}
&=&	\varprojlim 
\begin{bmatrix}	1&2\\ 1&1	\end{bmatrix}^{-1} 
\begin{bmatrix}	{R_{n+1}'}^t	\\	{R_n'}^t	\end{bmatrix}		\\
&=&	\varprojlim
\begin{bmatrix}	0	\\	{R_{n+1}'}^t	-	{R_n'}^t	\pmod{ (1+x)^{p^n}-1}	\end{bmatrix}.
\end{eqnarray*}

Note that if we were to define $\bfL_{\alpha}$ as in \cite{Kim}, then $\bfL_{1}=g_{tor}\times \det [ f_1, f_1']=0$, $\bfL_{2}=g_{tor}\times \det [ f_2, f_2']=0$, which are not meaningful. 
\end{example}

\vspace{2mm}

Similarly with \cite[Proposition~4.29]{Kim}, the next proposition follows from \cite{Perrin-Riou}~Lemme~5.2, and the standard Iwasawa theory arguments, and this proves the first part of our main result, Theorem~\ref{Vienna}.

\vspace{2mm}


\begin{proposition}			\Label{ZeroGo}
Recall $s=\dim A \cdot [H\cdot F:\Q]$. If $\bfL_{\{ \alpha_k \}} \not=0$, then for some fixed $C$

\[ \corank_{\Zp} \Sel_p(\bfA'/H\cdot F_n) \leq \sum_{i=1}^s (p-1) \times \left( p^{n-1}+p^{n-2}+ \cdots+ p^{m_i} \right)+C\]
where $n-m_i = \lambda_i n +O(1)$ and $\lambda_i=v_p(\alpha_i)$.

\end{proposition}

\vspace{2mm}

Clearly, $\rank A'(H\cdot F_n)$ is not greater than $\corank_{\Zp} \Sel_p(\bfA'/H\cdot F_n)$ (and equal to it if the Shafarevich-Tate group is finite).

Also, as noted in \cite{Kim}, $\bfL_{\{ \alpha_k \}} \not=0$ if $\Sel_p(\bfA'/H\cdot F_n)^{\chi}$ is finite for any $n$ and any character $\chi$ of $\Gal(H\cdot F_n/H\cdot F)$.

\end{section}


\begin{section}{Certain hyperelliptic curves and their Jacobians}

In this section we consider the following types of hyperelliptic curves and investigate the bounds of ranks of their Jacobian varieties. We study the following setup which has a certain Iwasawa theory flavor. Let $p>3$ be a prime. For positive integers $m$ and $n$, let
$$C_n:y^2=x^{3p^n}+ax^{p^n}+b,$$ and
$$C_{m,n} : y^{2p^m}=x^{3p^n}+ax^{p^n}+b.$$
($C_{m,n}$ is not a hyperelliptic curve. For convenience, we will call it a ramified hyperelliptic curve.)
We compute the genera of $C_n$ and $ C_{m,n}$ which are the dimensions of their jacobian varieties.

\begin{lem} 
\Label{Tripoli}
Denoting the genus of a curve $C$ by $g_C$,
$$g_{C_n}=\frac{3p^n-1}{2},$$ and
$$g_{C_{m,n}}=\begin{cases}1-2p^m+\frac{3p^n(2p^m-1)+p^m}{2}, &\text{ if } m>n\\
1-2p^m+\frac{3p^n(2p^m-1)+2p^m-p^n}{2}, &\text{ if } m\leq n.\end{cases}$$
\end{lem}

\begin{proof}  By the Riemann-Hurwitz formula, for $C_n$, the degree of $C_n$ to $\mathbb{P}^1$ is $2$ and the ramification index $e$ is $2$ at $3p^n+1$ points including at infinity. Hence it satisfies that
$$2g_{C_n}-2=2(-2)+(3p^n+1).$$

For $C_{m,n}$,  the degree of $C_{m,n}$ to $\mathbb{P}^1$ is $2p^m$.
If $m \le  n$, there are $3p^n$ points not at infinity and they have ramification index $e=2p^m$. And the ramification at infinity has the index $e=2$.   So
$$2g_{C_{m,n}}-2=2p^m(-2)+(3p^n)(2p^m-1)+p^m.$$

If $m > n$, similarly there are $3p^n$ points not at infinity and they have ramification index $e=2p^m$. And the ramification at infinity has the index $e = 2p^{m-n}$, so
$$2g-2=2p^m(-2)+(3p^n)(2p^m-1)+2p^m - p^n.$$
\end{proof}

Next we consider the endomorphism ring $\End(A)$ of an abelian variety $A$.

An abelian variety is isogenous to a product of simple abelian varieties. If $A$ is a simple abelian variety, then $\End(A)\otimes_{\Z}\Q$ is a division ring. So, if $A=A_1^{n_1}\times \cdots \times A_k^{n_k}$, where $A_k$'s are simple and not isogenous to each other, then
$$\End(A)\otimes_{\Z}\Q=M_{n_1}(D_1)\oplus\cdots \oplus M_{n_k}(D_k),$$
for some division algebras $D_i$.

This only gives a rough upper bound for the size of the endomorphism ring. However, in the case of the Jacobian varieties of $C_n$ and $C_{m,n}$, there is an obvious lower bound for the endomorphism rings. 

Suppose $C_n, C_{m,n}$ are defined over $\Q(\zeta_{p^n})$ and $\Q(\zeta_{p^m},\zeta_{p^n})$ respectively.
Let $J_n$ (defined over $\Q(\zeta_{p^n})$) be the Jacobian abelian variety of $C_n$, and $J_{m,n}$ (defined over $\Q(\zeta_{p^m}, \zeta_{p^n})$) be the Jacobian abelian variety of $C_{m,n}$.

Then, $\End_{\Q(\zeta_{p^n})} J_n$ contains $\Z[\zeta_{p^n}]$ induced by the automorphisms of $C_n$ given by $(x,y)\mapsto (\zeta x, y)$ for any $\zeta$ satisfying $\zeta^{p^n}=1$, and $\End_{\Q(\zeta_{p^m}, \zeta_{p^n})} J_{m,n}$ contains $\Z[\zeta_{p^m}] \times \Z[\zeta_{p^n}]$ induced by the automorphisms $(x,y)\mapsto (\zeta x, \zeta' y)$ for any $\zeta$ and $\zeta'$ satisfying $\zeta^{p^n}=1$ and $\zeta'^{p^m}=1$.


For the rest of the section, the abelian variety $A/F$ is either $J_n$ or $J_{m,n}$. If $A=J_n$, then the field $F$ is $\Q(\zeta_{p^n})$, and if $A=J_{m,n}$, then $F$ is $\Q(\zeta_{p^n}, \zeta_{p^m})$. In all cases, $A^\vee/F$ is the dual abelian variety of $A$.

Let $F_{\infty}=\Q(\zeta_{p^{\infty}})$.

\vspace{2mm}

\begin{assumption}		\Label{Liaodong}
$A/F$ has good reduction at the (unique) prime $\mathfrak p$ of $F$ above $p$.
\end{assumption}

\vspace{2mm}

At the moment, we do not know whether Assumption~\ref{Liaodong} is true for $J_n$, $J_{m,n}$, or their dual abelian varieties in general. We can only hope that it is often true. As mentioned earlier, we expect that the final result is probably true regardless. We hope that a proper theory for abelian varieties with bad reduction (especially unstable reduction) will be developed soon.

We let $H$ be a number field so that every prime of $H$ above $p$ is unramified over $H/\Q$. Some choice of $H$ might be more illuminating than others. For example, we may choose $H$ so that $A$ has ``complex multiplication'' over $H$: For instance, we may choose $H=\Q(i)$ when $p$ is an odd prime, and $b=0$.

\vspace{2mm}

\begin{definition}
Recall that $A^{\vee}$ is the dual abelian variety of $A$. Also recall

\[ \mfA= \cup A[p^n], \quad \mfA'=\cup A^\vee[p^n], \] where $A^\vee $ is the dual variety  of $A$.
(In other words, $\mfA'=\Hom(\mfA, \Zp(1))$.)

Also, recall $\Gamma=\Gal(H \cdot F_{\infty}/H \cdot F)$, and $\Lambda=\Zp[[\Gamma]]$ which we identify with $\Zp[[x]]$.
\end{definition}

\vspace{2mm}

\begin{assumption}			
\Label{Tyrol}

Recall the relative Selmer group $\Selr$ from Section~\ref{Turin}. We assume $\Selr(\bfA'/H \cdot F_{\infty})$ has $\Lambda$-corank $s=\dim A \cdot [H \cdot F:\Q]$.

\end{assumption}

\vspace{2mm}

As mentioned in Section~\ref{Turin}, Assumption~\ref{Tyrol} is true if $\Sel_p(\bfA/H\cdot F_n)^{\chi}$ is finite for any $n\geq 0$ and any character $\chi$ of $\Gal(H\cdot F_n / H\cdot F)$.

As in the discussion after Notation~\ref{AFG}, for each prime $\mfp$ of $F$ above $p$, define the Dieudonne module $M_{\mfp}$ associated to $A/F_{\mfp}$. Also, as we did in the same discussion, construct local points $\{ Q_{i,n}\}_n$ for $i=1,\ldots, \dim A \cdot [HF:\Q]$, each associated to a monic irreducible polynomial $q_i(x) \in \Zp[x]$.

The polynomials $q_1,\ldots, q_{s}$ depend on the choice of the above-mentioned local points. However, as we have demonstrated earlier, all of them are irreducible divisors of $\det (x\cdot 1-\bfF|M_{\mfp})$ for some $\mfp$, and every irreducible (polynomial) divisor of $\det (x\cdot 1-\bfF|M_{\mfp})$ of any $\mfp$ is represented by them.

Then, by Proposition~\ref{ZeroGo} and Lemma~\ref{Tripoli}, we obtain Theorem~\ref{Vienna}.


\end{section}

\end{nouppercase}


\end{document}